\documentclass[12pt]{amsart}

\usepackage{amsmath,amsfonts,amssymb,amscd,amsthm,latexsym}

\usepackage{pstricks,pst-node,pst-plot}
\usepackage{graphicx}

\newtheorem{theorem}{Theorem}[section]

\newtheorem{corollary}[theorem]{Corollary}

\newtheorem{lemma}[theorem]{Lemma}

\newtheorem{proposition}[theorem]{Proposition}
\newtheorem{remark}[theorem]{Remark}

\begin{document}

\title[Regular $S$-acts with primitive normal and antiadditive theories]
{Regular $S$-acts with primitive normal and antiadditive theories}
\subjclass{20 M 10}

\noindent 510.67:512.56 
\date{\today}
\author{Stepanova A.A. ${}^1$, Baturin G.I.} \footnotetext[1]
{This research was supported by RFBR (grant 09-01-00336-a)}
\address{Institute of Mathematics and Computer Science\\Far East State
University\\Vladivostok\\Russia} \email{stepltd@mail.ru}
\address{Institute of Mathematics and Computer Science\\Far East State
University\\Vladivostok\\Russia} \email{gbaturin@list.ru}

\begin{abstract} In this work we
investigate the commutative monoids over which the axiomatizable
class of regular $S$-acts is primitive normal and antiadditive. We
prove that the primitive normality of an axiomatizable class of
regular $S$-acts over the commutative monoid $S$ is equivalent to
the antiadditivity of this class and it is equivalent to a
linearly order of semigroup $R$ such that an $S$-act $_SR$ is a
maximum under the inclusion regular subact of $S$-act $_SS$.

\end{abstract} \maketitle

\sloppy  \vskip 1cm

\section{Introduction}
In \cite{St} the primitive normal, primitive connected and
additive theories of $S$-acts are studied. In particular it is
proved that a class of all $S$-acts is primitive normal if and
only if $S$ is a linearly ordered monoid. In  \cite {St1} on a
language of a structure of primitive equivalences there are
described $S$-acts with primitive normal, additive and
antiadditive theories. It is shown that the class of all $S$-acts
is antiadditive only for a linearly ordered monoid $S$, that is
the class of all $S$-acts is antiadditive if and only if  this
class is primitive normal. In this work we investigate the
commutative monoids over which the axiomatizable class of regular
$S$-acts is primitive normal and antiadditive.  We prove that the
primitive normality of an axiomatizable class of regular $S$-acts
over the commutative monoid $S$ is equivalent to the
antiadditivity of this class and it is equivalent to a linearly
order of semigroup $R$ such that an $S$-act $_SR$ is a maximum
under the inclusion regular subact of $S$-act $_SS$.

Let $T$ be a complete first order theory of a language $L$. We fix
some large much saturated model $\mathcal {C}$ of $T$ and we
suppose that all considered models of the theory are its
elementary submodels. All elements, tuples of elements and sets
will be taken from $\mathcal {C}$. The tuples of elements $\langle
a_0,\ldots, a_{n-1}\rangle$ and variables $\langle
x_0,\ldots,x_{n-1}\rangle$ will be denoted by $\bar a$ and $\bar
x$ accordingly. Let $ \bar s =\langle s_0, \ldots, s _ {n-1}
\rangle $ and $\bar t=\langle t_0,\ldots,t_{k-1}\rangle$ be the
tuples of variables or elements, $A$ be a set. We will often write
$\bar s\in A$ instead $s_0,\ldots,s_{n-1}\in A$, $\bar s(i)$
instead $s_i$, $\exists\bar s$ instead $\exists s_0\ldots\exists
s_{n-1}$. The set $\{s_0,\ldots,s_{n-1},t_0,\ldots,t_{k-1}\}$ we
will denote by $\bar s\cup\bar t$. We will denote the length of a
tuple $\bar s$ by $|\bar s|$, i.e. $|\bar s|=n$. If $\Phi(\bar
x,\bar y)$ is a formula of language $L$, $\mathcal{A}$ is a model
of the theory $T$, $\bar a$ is a tuple of elements from
$\mathcal{A}$ and $|\bar a|=|\bar y|$, then $\Phi(\mathcal{A},\bar
a)$ will denote the set $\{\bar b\mid \mathcal{A}\models\Phi(\bar
b,\bar a)\}$.

The formula of a form
$$\exists \bar x (\Phi_0\wedge\cdots\wedge\Phi_k),$$
where $\Phi_i$ are the atomic formulas $(0\leqslant i\leqslant
k)$, is called \em a primitive formula.\em

 Let $\Phi(\bar x,\bar y)$ be a primitive formula of language $L$,
$\bar a$, $\bar b$ be the tuples of elements and $|\bar a|=|\bar
b|=|\bar y|$. The set $\Phi(\mathcal{C},\bar a)$ is called \em a
primitive set\em. The sets $\Phi(\mathcal{C},\bar a)$ and
$\Phi(\mathcal{C},\bar b)$ are called \em the primitive copies.\em

A theory $T$ is called \em primitive normal \em if for each
primitive copies $X,Y$ we have $X=Y $ or $X\cap Y=\varnothing$. An
axiomatizable class of structures $K$ of language $L$ is called
\em primitive normal \em if the theory of this class is primitive
normal. It is known (see \cite {P}) that the Cartesian closed
stable class of structures is primitive normal.

An equivalence $\alpha$ on some set $X$ of $n$-tuples of elements
from $\mathcal{C}$, which is defined in $\mathcal{C}$ by some
primitive formula $\Phi(\bar x_1,\bar x_2)$, is called \em a
primitive equivalence\em. The domain $X$ of such equivalence
$\alpha$ is defined in $\mathcal{C}$ by primitive formula
$\Phi(\bar x,\bar x)$ and is denoted by $dom(\alpha)$. If $\bar
a\in X$ then $\alpha$-class which contains $\bar a$ will be
denoted by $\bar a/\alpha$.

A set $X$ is called \em $\Delta$-primitive \em if there exists a
family $S$ of primitive sets such that
 $$X=\bigcap\{Y\mid Y\in S\}.$$
A set of form $X=X^*/\alpha=\{\bar a/\alpha\mid\bar a\in X^*\}$,
where $X^*$ is $\Delta$-primitive set, $\alpha$ is primitive
equivalence and $X^*\subseteq dom(\alpha)$, is called \em a
generalized primitive set\em. A set $X^*$ is called \em a basis
\em and $\alpha$  is called \em a generative equivalence \em of
generalized primitive set $X$.

The theory $T$ is called \em antiadditive \em if it is primitive
normal and there is no infinite generalized primitive set which is
an Abelian group under the defined by primitive formula operation.
An axiomatizable class of structures $K$ of language $L$ is called
\em antiadditive \em if the theory of this class is antiadditive.

Let us remind some concepts from the theory of $S$-acts.
Throughout this paper $S$ will denote a monoid with identity $1$
and set of idempotents $E$. A structure $\langle A;s\rangle_{s\in
S}$ of the language $L_S=\{s\mid s\in S\}$ is \em a (left) $S$-act
\em if $s_1(s_2a)=(s_1s_2)a$ and $1a=a$ for all $s_1,s_2\in S$ and
$a\in A$. An $S$-act $\langle A;s\rangle_{s\in S}$ we will denote
by ${}_SA$. All $S$-acts, treated in the article, are left
$S$-acts.

Let $_SA$, $_SB$ be $S$-acts. We call $a\in A$ \em an act-regular
element \em if there exists a homomorphism
$\varphi:{}_SSa\longrightarrow {}_SS$ such that $\varphi(a)a=a$.
An $S$-act $_SA$ is called \em regular \em if all its elements are
act--regular. If for elements $a\in A$ and $b\in B$ there is an
isomorphism $f: {}_SSa\to {}_SSb$ such that $f(a)=b$ then we will
write ${}_SSa \mathrel{\widetilde{\to}} {}_SSb$.

Note that the union of all regular subacts of an $S$-act is also a
regular subact. The union of all regular subacts of an $S$-act
${}_SS$ we will denote by ${}_SR$. Hereinafter, we assume that
$R\neq \emptyset$.

A semigroup $T$ is called  \emph{linearly ordered} if for all
$a,b\in T$ either $Ta\subseteq Tb$ or $Tb\subseteq Ta$. A monoid
$S$ is called \emph{regularly linearly ordered} if  for all $a\in
R$ a semigroup $Sa$ is linearly ordered.

We will distinguish the symbols of set-theoretic inclusion
$\subset$ and $\subseteq.$

\section{Primitive normal Classes of Regular Acts}

We will use the following remark without references to it.

\begin {remark} \label {Se subsut Sf} For all $a\in S,\;e,f\in E$
we have

1) $aS\subseteq eS\Longleftrightarrow ea=a$;

2) $Sa\subseteq Se\Longleftrightarrow ae=a$.
\end{remark}

\begin{theorem}\label{prim norm iff}
An $S$-act $_SA$ is primitive normal if and only if for any
pairwise disjoint finite sets of indexes $I, J, K$, any
$s_i,l_j,r^1_k,r^2_k\in S$ ($i\in I, j\in J, k\in K$), $\bar a_1,
\bar a_2, \bar a_3\in A$, $|\bar a_1|=|\bar a_2|=|\bar a_3|=n$, if
\begin{equation}\tag{1}
_SA\models\bigwedge\limits_{i\in I}s_i\bar a_1(l_i)=s_i\bar
a_2(l_i) \wedge\bigwedge\limits_{j\in J}t_j\bar a_2(l_j)=t_j\bar
a_3(l_j) \wedge\bigwedge\limits_{k\in K}\bigwedge\limits_{m\in
\{1,2,3\}}r^1_k\bar a_m(l_k)=r^2_k\bar a_m(l_k),
\end{equation}
where $0\leqslant l_i,l_j,l_k\leqslant n-1$, then there exists
$\bar b\in A$ such that $|\bar b|=n$ and
\begin{equation}\tag{2}_SA\models\bigwedge\limits_{i\in I}s_i\bar a_3(l_i)=s_i\bar b(l_i)
\wedge\bigwedge\limits_{j\in J}t_j\bar b(l_j)=t_j\bar a_1(l_j)
\wedge\bigwedge\limits_{k\in K}r^1_k\bar b(l_k)=r^2_k\bar b(l_k).
\end{equation}
\end{theorem}

\begin{proof} \emph{Necessity}. Let $_SA$ be a primitive normal $S$-act and
(1) hold for some pairwise disjoint finite sets of indexes $I, J,
K$, some $s_i,l_j,r^1_k,r^2_k\in S$ ($i\in I, j\in J, k\in K$) and
$\bar a_1, \bar a_2, \bar a_3\in A$, $|\bar a_1|=|\bar a_2|=|\bar
a_3|=n$. We put
$$\Phi(\bar x,\bar y)\leftrightharpoons\exists\bar u(\bigwedge\limits_{i\in I}s_i\bar x(l_i)=s_i\bar u(l_i)
\wedge\bigwedge\limits_{j\in J}t_j\bar u(l_j)=t_j\bar y(l_j)
\wedge $$
$$\wedge \bigwedge\limits_{k\in
K}r^1_k\bar x(l_k)=r^2_k\bar x(l_k)\wedge\bigwedge\limits_{k\in
K}r^1_k\bar u(l_k)=r^2_k\bar u(l_k)\wedge\bigwedge\limits_{k\in
K}r^1_k\bar y(l_k)=r^2_k\bar y(l_k)),$$
 where $|\bar x|=|\bar
u|=|\bar y|=n$. By a condition $\bar a_1\in\Phi(A,\bar a_1)$ and
$\bar a_1,\bar a_3\in\Phi(A,\bar a_3)$. Since the $S$-act $_SA$ is
primitive normal then $\bar a_3\in\Phi(A,\bar a_1)$ that is the
condition (2) holds.

\emph{Sufficiency}. Let $\Psi(\bar x,\bar y)$ be a primitive
formula,
$$\Psi(\bar x,\bar y)\leftrightharpoons\exists\bar u\Theta(\bar x,\bar y,\bar u),$$
where $$\Theta(\bar x,\bar y,\bar u)\leftrightharpoons\bigwedge
\limits_{\langle i,j\rangle\in I_1}s^1_i\bar x(l_i)=t^1_j\bar
x(l_j) \wedge\bigwedge\limits_{\langle i,j\rangle\in I_2}s^2_i\bar
x(l_i)=t^2_j\bar u(l_j) \wedge \bigwedge\limits_{\langle
i,j\rangle\in I_3}s^3_i\bar u(l_i)=t^3_j\bar u(l_j)\wedge$$
 $$\wedge\bigwedge\limits_{\langle i,j\rangle\in I_4}s^4_i\bar
u(l_i)=t^4_j\bar y(l_j)\wedge\bigwedge\limits_{\langle
i,j\rangle\in I_5}s^5_i\bar y(l_i)=t^5_j\bar
y(l_j)\wedge\bigwedge\limits_{\langle i,j\rangle\in I_6}s^6_i\bar
y(l_i)=t^6_j\bar x(l_j),$$
 $s_i^k,t_j^k\in S$ ($\langle i,j\rangle\in I_k$, $1\leqslant k\leqslant
 6$).
 We will show that $_SA$ is a primitive normal $S$-act. Suppose that $\bar
a_2\in\Psi(A,\bar a_1)$, $\bar a_2,\bar a_3\in\Psi(A,\bar a_4)$
and $\bar a_1,\bar a_2,\bar a_3,\bar a_4\in A$. Then
\begin{equation}\tag{3}_SA\models\Theta(\bar a_2,\bar a_1,\bar b_{21})\wedge\Theta(\bar a_2,\bar a_4,\bar b_{24})
\wedge\Theta(\bar a_3,\bar a_4,\bar b_{34})
\end{equation}
for some $\bar b_{21},\bar b_{24},\bar b_{34}\in A$. It is enough
to show that $\bar a_3\in\Psi(A,\bar a_1)$. From (3) we have
\begin{equation}\tag{4}
_SA\models\bigwedge\limits_{\langle i,j\rangle\in I_1}s^1_i\bar
a_3(l_i)=t^1_j\bar a_{3}(l_j)\wedge\bigwedge\limits_{\langle
i,j\rangle\in I_5}s^5_i\bar a_1(l_i)=t^5_j\bar a_1(l_j);
\end{equation}
\begin{equation}\tag{5}
_SA\models\bigwedge\limits_{\langle i,j\rangle\in I_2}t^2_j\bar
b_{21}(l_j)=s^2_i\bar a_2(l_i)=t^2_j\bar b_{24}(l_j);
\end{equation}
\begin{equation}\tag{6}
_SA\models\bigwedge\limits_{\langle i,j\rangle\in I_4}s^4_i\bar
b_{24}(l_i)=t^4_j\bar a_4(l_j)=s^4_i\bar b_{34}(l_i);
\end{equation}
\begin{equation}\tag{7}
_SA\models\bigwedge\limits_{\langle i,j\rangle\in I_3}s^3_i\bar
b_{km}(l_i)=t^3_j\bar b_{km}(l_j);
\end{equation}
\begin{equation}\tag{8}
_SA\models\bigwedge\limits_{\langle i,j\rangle\in I_6}s^6_i\bar
a_{1}(l_i)=t^6_j\bar a_{2}(l_j)=s^6_i\bar a_{4}(l_i)=t^6_j\bar
a_{3}(l_j)
\end{equation}
for all $\langle k,m\rangle\in \{\langle 2,1\rangle,\langle
2,4\rangle,\langle 3,4\rangle\}$. The conditions of Theorem and
(5), (6), (7) imply
 $$
_SA\models\bigwedge\limits_{\langle i,j\rangle\in I_2}t^2_j\bar
b_{34}(l_j)=t^2_j\bar b(l_j)\wedge\bigwedge\limits_{\langle
i,j\rangle\in I_4}s^4_i\bar b(l_i)=s^4_i\bar
b_{21}(l_i)\wedge\bigwedge\limits_{\langle i,j\rangle\in
I_3}s^3_i\bar b(l_i)=t^3_j\bar b(l_j)
 $$
for some $\bar b\in A$. Using this and (3) we have
 $$_SA\models\bigwedge\limits_{\langle i,j\rangle\in I_2}s^2_i\bar
a_{3}(l_i)=t^2_j\bar b(l_j)\wedge\bigwedge\limits_{\langle
i,j\rangle\in I_4}s^4_i\bar b(l_i)=t^4_j\bar
a_{1}(l_j)\wedge\bigwedge\limits_{\langle i,j\rangle\in
I_3}s^3_i\bar b(l_i)=t^3_i\bar b(l_i).
 $$
  This one, (4) and (8) imply $_SA\models\Theta(\bar a_3,\bar a_1,\bar b)$, that is $\bar
a_3\in\Psi(A,\bar a_1)$. So we have proved that $_SA$ is a
primitive normal $S$-act.
\end{proof}

\begin{proposition}\label{act-reg element} \cite{KKM} Let  $_SA$ be an $S$-act, $a\in A$.
An element $a$ is act-regular if and only if there exists an
idempotent $e\in R$ such that ${}_SSa \mathrel{\widetilde{\to}}
{}_SSe$.
\end{proposition}

\begin{proposition}\label{R=cup Re}\cite{MOPS} If the class ${\mathfrak R}$
of regular $S$-acts is axiomatizable then $R=\bigcup\{e_iR\mid
1\leqslant i\leqslant n \}$ for some $n\geqslant 1$, $e_i\in R$,
$e^2_i=e_i$ $(1\leqslant i\leqslant n)$.
\end{proposition}

\begin{lemma}\label{if R prim norm then S reg lin ord} Let the class $\mathfrak {R}$
of regular $S$-acts is axiomatizable and primitive normal. Then
$R$ is a regularly linearly ordered monoid.
\end{lemma}

\begin{proof}
Assume that $Sc\not\subseteq Sb$ and $Sb\not\subseteq Sc$ for some
$b,c\in Sa$ и $a\in R$. There exists $e\in E$ such that ${}_SSa
\mathrel{\widetilde{\to}} {}_SSe$. Let $_SSe_i$ $(1\leqslant
i\leqslant 3)$ be the pairwise disjoint copies of $S$-act $_SSe$,
$\Theta$ be a congruence of $S$-act
$\bigsqcup\limits_{i=1}^3{_SS}e_i$ generated by $\{\langle
ce_1,ce_2\rangle,\langle be_2,be_3\rangle\}$. Let $_SA$ denote an
$S$-act $\bigsqcup\limits_{i=1}^3{_SSe}_i/\Theta$, $d/\Theta$
denote a $\Theta$-class of $d\in \bigsqcup\limits_{i=1}^3{Se}_i$.
Let
$$\Phi(x,y)\rightleftharpoons\exists u(bu=x\wedge cu=y).$$
Hence
$$_SA\models\Phi(be_1/\Theta,ce_1/\Theta)\wedge\Phi(be_2/\Theta,ce_1/\Theta)
\wedge\Phi(be_2/\Theta,ce_3/\Theta).$$
 Since the class $\mathfrak {R}$
is primitive normal then
$_SA\models\Phi(be_1/\Theta,ce_3/\Theta)$. Let $u^0\in
\bigsqcup\limits_{i=1}^3{Se}_i$ such that $_SA\models
bu^0/\Theta=be_1/\Theta\wedge cu^0/\Theta=ce_3/\Theta$. Then
$be_1/\Theta,ce_3/\Theta\in Su^0/\Theta$. But
$be_1/\Theta=\{be_1\}$, $ce_3/\Theta=\{ce_3\}$ and $S$-acts
$_SSe_1$, $_SSe_3$ do not intersect. A contradiction.
\end{proof}

\begin{lemma}\label{if R prim norm then Se and Sf comparable} Let $S$
be a commutative monoid, the class $\mathfrak {R}$ of regular
$S$-acts is axiomatizable and primitive normal. Then for any
idempotents $e,f\in R$ either $Se\subseteq Sf$ or $Sf\subseteq
Se$.
\end{lemma}

\begin{proof}
Suppose that $Se\not\subseteq Sf$ and $Sf\not\subseteq Se$ for
some idempotents $e,f\in R$. Note that $ef\in E$. If $Sef=Sf$ then
$ef=fef=f$, that is $Sf\subseteq Se$, a contradiction. Hence
$Sef\subset Sf$. Similarly $Sef=Sfe\subset Se$. Let
$$\Phi(x,y)\rightleftharpoons\exists u(eu=ex\wedge fu=fy).$$
So
 $$_SR\models\Phi(e,e)\wedge\Phi(f,f)
\wedge\Phi(f,e).$$ Since the class $\mathfrak {R}$ is primitive
normal then $_SR\models\Phi(e,f)$. Let $u^0\in R$ such that
$_SR\models e=eu^0\wedge f=fu^0$. Then $e,f\in Su^0$. Hence by
Lemma \ref{if R prim norm then S reg lin ord} either $Se\subseteq
Sf$ or $Sf\subseteq Se$, but that contradicts to assumption.
\end{proof}

\begin{theorem}\label{R prim norm iff}
Let $S$ be a commutative monoid and the class $\mathfrak {R}$ of
regular $S$-acts is axiomatizable. The class $\mathfrak {R}$ is
primitive normal if and only if a semigroup $R$ is linearly
ordered.
\end{theorem}

\begin{proof} \emph{Necessity}. Let the class $\mathfrak {R}$ is primitive normal.
We will show that a semigroup $R$ is linearly ordered. Since the
class $\mathfrak {R}$ is axiomatizable then by Proposition
\ref{R=cup Re} $R=\bigcup\{e_iR\mid 1\leqslant i\leqslant m \}$
for some $m\geqslant 1$ and idempotents $e_i\in R$ $(1\leqslant
i\leqslant m)$. As $sf=fsf\in fR$ for all $s\in S$ then $fR=Sf$,
where $f\in R\cap E$. So in view of commutativity of a  monoid $S$
and by Lemma \ref{if R prim norm then Se and Sf comparable}
$R=eR=Se$ for some idempotent $e\in R$. Thus by Lemma \ref{if R
prim norm then S reg lin ord} $R$ is a linearly ordered semigroup.

\emph{Sufficiency}. Let $_SA\in \mathfrak {R}$, $I, J, K$ be the
pairwise disjoint finite sets of indexes, $s_i,l_j,r^1_k,r^2_k\in
S$ ($i\in I, j\in J, k\in K$), $\bar a_1, \bar a_2, \bar a_3\in
A$, $|\bar a_1|=|\bar a_2|=|\bar a_3|=n$ and (1) holds, where
$0\leqslant l_i,l_j,l_k\leqslant n-1$. By Proposition \ref{act-reg
element} there exists the tuples of idempotents $\bar e_1,\bar
e_2,\bar e_3\in R$ such that $|\bar e_1|=|\bar e_2|=|\bar e_3|=n$
and ${}_SS\bar a_i(j) \mathrel{\widetilde{\to}} {}_SS\bar e_i(j)$
for all $i,j$, $0\leqslant i\leqslant 3$, $0\leqslant j\leqslant
n-1$. Then $\bar a_i(j)=\bar e_i(j)\bar a_i(j)$ for all $i,j$,
$0\leqslant i\leqslant 3$, $0\leqslant j\leqslant n-1$.

We will construct a tuple $\bar b$ such that (2) holds.

Let us fix $l\in \{0,1,\ldots,n-1\}$. We put $I_l=\{i\in I\mid
l_{i}=l\}$, $J_l=\{j\in J\mid l_{j}=l\}$. Let $1\leqslant
k\leqslant 3$. Since by the condition the set $\{Sd\mid
Sd\subseteq Se_k(l)\}$ is linearly ordered under the inclusion,
then there exist $s^k,t^k\in S$ such that
$Ss^ke_k(l)=\max\{Ss_ie_k(l)\mid i\in I_l\}$ and
$St^ke_k(l)=\max\{St_je_k(l)\mid j\in J_l\}$. Hence for all $i\in
I_l$ and $j\in J_l$ there are $r_i^k\in S$ and $r_j^k\in S$ such
that $s_i\bar e_k(l)=r_i^ks^k\bar e_k(l)$ and $t_j\bar
e_k(l)=r_j^kt^k\bar e_k(l)$, so $s_i\bar a_k(l)=r_i^ks^k\bar
a_k(l)$ and $t_j\bar a_k(l)=r_j^kt^k\bar a_k(l)$.

Assume that
\begin{equation}\tag{9}
S\bar e_1(l)\subseteq S\bar e_2(l)\mbox { and } S\bar
e_3(l)\subseteq S\bar e_2(l),
\end{equation}
that is $\bar e_k(l)=\bar e_2(l)\bar e_k(l)$ and $\bar a_k(l)=\bar
e_2(l)\bar a_k(l)$ for all $k$, $1\leqslant k \leqslant 3$. Since
the semigroup $S\bar e_2(l)$ is linearly ordered, without loss of
generality we can suppose that $S\bar e_1(l)\subseteq Se_3(l)$,
i.e. $\bar e_1(l)=\bar e_1(l)\bar e_3(l)$. Assume that $St^2\bar
e_2(l)\subseteq Ss^2\bar e_2(l)$, i.e. $t^2\bar e_2(l)=r^2s^2\bar
e_2(l)$ for some $r^2\in S$. Let $j\in J_l$. Then $t_j\bar
e_2(l)=r_j^2t^2\bar e_2(l)=r_j^2r^2s^2\bar e_2(l)$ and $t_j\bar
a_2(l)=r_j^2r^2s^2\bar a_2(l)$. So in view (1) we have
$$t_j\bar a_3(l)=t_j\bar a_2(l)=r_j^2r^2s^2\bar a_2(l)
 =r_j^2r^2s^2\bar a_1(l)=r_j^2r^2s^2\bar e_2(l)\bar a_1(l)=t_j\bar e_2(l)\bar
a_1(l)=t_j\bar a_1(l),$$ that is $t_j\bar a_3(l)=t_j\bar a_1(l)$.
We put $\bar b(l)=\bar a_3(l)$. If $Ss^2\bar e_2(l)\subseteq
St^2\bar e_2(l)$ then in the same way we have $\bar b(l)=\bar
a_1(l)$.

Let (9) be wrong. Then without loss of generality we can suppose
that
$$
S\bar e_3(l)\subseteq S\bar e_1(l)\mbox { and } S\bar
e_2(l)\subseteq S\bar e_1(l),
$$
that is $\bar e_k(l)=\bar e_1(l)\bar e_k(l)$ and $\bar a_k(l)=\bar
e_1(l)\bar a_k(l)$ for all $k$, $1\leqslant k \leqslant 3$.

Assume that $St^1\bar e_1(l)\subseteq Ss^1\bar e_1(l)$, that is
$t^1\bar e_1(l)=r^1s^1\bar e_1(l)$ for some $r^1\in S$. Let $j\in
J_l$. Then $t_j\bar e_1(l)=r_j^1t^1\bar e_1(l)=r_j^1r^1s^1\bar
e_1(l)$ and $t_j\bar a_1(l)=r_j^1r^1s^1\bar a_1(l)$. Hence using
(1) we get
$$t_j\bar a_3(l)=t_j\bar a_2(l)=t_j\bar e_1(l)\bar a_2(l)=r_j^1r^1s^1\bar e_1(l)\bar a_2(l)
=r_j^1r^1s^1\bar a_2(l)=r_j^1r^1s^1\bar a_1(l)=t_j\bar a_1(l),$$
that is $t_j\bar a_3(l)=t_j\bar a_1(l)$. We set $\bar b(l)=\bar
a_3(l)$.

Suppose that $Ss^1\bar e_1(l)\subseteq St^1\bar e_1(l)$, that is
$s^1\bar e_1(l)=r_1t^1\bar e_1(l)$ for some $r_1\in S$. Let $i\in
I_l$. Then $s_i\bar e_1(l)=r_i^1s^1\bar e_1(l)=r_i^1r_1t^1\bar
e_1(l).$ So using (1) we get
$$s_i\bar a_3(l)=s_i\bar e_1(l)\bar a_3(l)=r_i^1r_1t^1\bar e_1(l)\bar a_3(l)
=r_i^1r_1t^1\bar e_1(l)\bar a_2(l) =$$
 $$=s_i\bar e_1(l)\bar
a_2(l)=s_i\bar a_2(l)=s_i\bar a_1(l),$$
 that is $s_i\bar
a_3(l)=s_i\bar a_1(l)$. We set $\bar b(l)=\bar a_1(l)$.

Therefore the tuple $\bar b$ such that (2) holds is construct.
\end{proof}

The following \em example \em shows that the condition of
commutativity of monoid  $S$ in Theorem \ref{R prim norm iff} is
essentially.

Let $S=\{e_1,e_2\}\cup T\cup\{1\}$, where $T$ is a semigroup with
$\{a,b\}$ generators  and $ab^2=ab$, $ba^2=ba$ defining
relationships. Binary operation on $S$ is defined in the following
way:  $se_i=e_i$, $e_it=e_i$ for all $i\in \{1,2\}$, $s\in S$,
$t\in T$; $1$ is an unit element. It is easy to check that $S$ is
a monoid under the operation and $E=\{e_1,e_2,1\}$. Note that
$Se_i=\{e_i\}$ for all $i\in \{1,2\}$. Since $ba^2=ba$ and $ba\neq
b$ then the assertion $_SSa\mathrel{\widetilde{\to}} _SS\cdot 1$
is false. Since $a^2\neq a$ and $ae_i=e_i$ then the assertion
$_SSa\mathrel{\widetilde{\to}} _SS e_i$ is false for all $i$,
$i\in \{1,2\}$. In the same way it is shown the falsity of the
assertion $_SSb\mathrel{\widetilde{\to}} _SS\cdot 1$ and
$_SSb\mathrel{\widetilde{\to}} _SS e_i$ ($i\in \{1,2\}$).

So  $R=\{e_1,e_2\}$. It is clear that the semigroup $R$ is not
linearly order. For all $S$-act $_SA\in \mathfrak {R}$, $a\in A$
and $s\in S$ we have $sa=a$, that is any regular $S$-act is
represent as a coproduct of one-element $S$-acts. Hence the class
$\mathfrak R$ is axiomatizable and primitive normal.

\section{Antiadditive Classes of Regular Acts}

The axiomatizable classes of regular $S$-acts were investigated in
\cite{St2}. Particularly in that work there was proved the
following proposition.

\begin{proposition}
If the class ${\mathfrak R}$ of regular $S$-acts is axiomatizable
then $R=\bigcup\{e_iR\mid 1\leqslant i\leqslant n \}$ for some
$n\geqslant 1$, $e_i\in R$, $e^2_i=e_i$ $(1\leqslant i\leqslant
n)$.
\end{proposition}
This statement implies
\begin{corollary} \label{R is the cup of eR}
If the class ${\mathfrak R}$ of regular $S$-acts is axiomatizable,
monoid $S$ is commutative and  $R$ will be a linearly order
semigroup then $R=eR$ for some idempotent $e\in R$.
\end{corollary}
Throughout $T$ will denote a theory of the axiomatizable primitive
normal class ${\mathfrak R}$ of regular $S$-acts, $S$ will be a
commutative monoid and  $R$ is a linearly order semigroup.

A proof of following Lemma is a modification of a proof of Lemma 2
in \cite{St1}.

\begin{lemma}\label{conj atom form} Let $\Phi(x_0,\bar x)$ be a conjunction of
atomic formulas, $\bar x= \langle x_1,\ldots,x_n\rangle$. Then
there is a formula $\Psi(\bar x)$, which is a conjunction of
atomic formulas, $s,t\in S$  and $i$, $0\leqslant i\leqslant n$,
such that in theory $T$
 $$\Phi(x_0,\bar x)\equiv\Psi(\bar
x)\wedge tx_i=sx_0.$$
\end{lemma}

\begin{proof} By Corollary \ref{R is the cup of eR} there exists an idempotent
$e\in R$ such that $R=eR$.
Then
\begin{equation}\tag{10}
T\vdash\forall x(x=ex).
\end{equation}
Let $\Phi(x_0,\bar x)$ be a conjunction of atomic formulas, $\bar
x= \langle x_1,\ldots,x_n\rangle$. We will prove Lemma by the
induction on a number $k$ of atomic subformulas of formula
$\Phi(x_0,\bar x)$ containing a variable $x_0$. Suppose that
 $$\Phi(x_0,\bar x)\rightleftharpoons\Psi_1(x_0,\bar
x)\wedge t_1x_i=s_1x_0,$$
 where $\Psi_1(x_0,\bar x)$ is a conjunction of atomic formulas, $s_1,t_1\in
 S$, $0\leqslant i\leqslant n$. On the
suggestion of the induction
$$\Psi_1(x_0,\bar x)\equiv\Psi_2(\bar
x)\wedge t_2x_j=s_2x_0$$ for some formula $\Psi_2(\bar x)$, which
is a conjunction of atomic formulas, some $s_2,t_2\in S$ and $j$,
$0\leqslant j\leqslant n$. In view of linearly order of a
semigroup $R$ we have either $Ss_1e\subseteq Ss_2e$ or
$Ss_2e\subseteq Ss_1e$. Let for example $Ss_1e\subseteq Ss_2e$.
Then there exists $r\in S$ such that $s_1e=rs_2e$. Hence in view
of (10) we have
$$\Phi(x_0,\bar x)\equiv\Psi_2(\bar
x)\wedge t_2x_j=s_2x_0\wedge t_1x_i=s_1x_0\equiv $$
 $$\equiv\Psi_2(\bar x)\wedge rt_2x_j=t_1x_i\wedge t_2x_j=s_2x_0.$$
Lemma is proved.
\end{proof}

The proof of following Lemma coincides exactly with the proof of
Lemma 3 in \cite{St1}.

\begin{lemma}\label{conj prim form} Let $\Phi(\bar x)$ be not always-false
primitive formula, $\bar x= \langle x_1,\ldots,x_n\rangle$. Then
there exists the formula $\Phi_0(\bar x)$, which is a conjunction
of atomic formulas, and the primitive formulas $\Phi_i(x_i)$,
$1\leqslant i\leqslant n$, such that in theory $T$
$$\Phi(\bar x)\equiv\Phi_0(\bar
x)\wedge \bigwedge_{1\leqslant i\leqslant n}\Phi_i(x_i).$$
\end{lemma}

\begin{lemma}\label{eliminac} Let $\bar a\in\mathcal{C}$,
$\Phi(\bar x,\bar y,\bar z,\bar a)$ be a primitive formula which
defines on the infinite generalized primitive set $X$ a binary
operation  $+:$ $\bar x+\bar y=\bar z$. Then the set $X$ is not a
group under this operation.
\end{lemma}

\begin{proof} Let $\bar a\in\mathcal{C}$,
$\Phi(\bar x,\bar y,\bar z,\bar a)$ be a primitive formula
defining a structure of group relative to a binary operation $+$,
$X^*$ be a basis, $\alpha$ be a generative equivalence of the
generalized primitive set $X$, $|\bar a|=|\bar u|$. By Lemma
\ref{conj prim form} in the theory $T$
 $$\Phi(\bar x,\bar y,\bar z,\bar u)\equiv
 \Phi_0(\bar x,\bar y,\bar z,\bar u)\wedge \Phi_1(\bar x)\wedge \Phi_2(\bar y)
 \wedge \Phi_3(\bar z)\wedge \Phi_4(\bar u)$$
for some formulas $\Phi_0(\bar x,\bar y,\bar z,\bar u)$,
$\Phi_1(\bar x)$, $\Phi_2(\bar y)$, $\Phi_3(\bar z)$, $\Phi_4(\bar
u)$, where $\Phi_0(\bar x,\bar y,\bar z,\bar u)$ is a conjunction
of atomic formulas, $\Phi_1(\bar x)$, $\Phi_2(\bar y)$,
$\Phi_3(\bar z)$, $\Phi_4(\bar u)$ are the primitive formulas. By
Lemma \ref{conj atom form} there exist the formula $\Psi(\bar
x,\bar y,\bar u)$, $t_i,s_i\in S$ and $\bar w= \langle
w_1,\ldots,w_n\rangle$, where $w_i\in \bar x\cup\bar y\cup\bar
z\cup\bar u$, $1\leqslant i\leqslant n$, such that $\Psi(\bar
x,\bar y,\bar u)$ is a conjunction of atomic formulas and in the
theory $T$
$$
\Phi_0(\bar x,\bar y,\bar z,\bar u)\equiv\Psi(\bar x,\bar y,\bar
u)\wedge\Theta(\bar x,\bar y,\bar z,\bar u),
$$
where
$$\Theta(\bar x,\bar y,\bar z,\bar u)\rightleftharpoons\bigwedge_{1\leqslant i\leqslant n}
 t_iz_i=s_iw_i.$$
Let $\bar b,\bar c\in X^*$, $\bar 0/\alpha$ be a null element of
the group $X$. Suppose that $t_iz_i=s_ix_j$ is an atomic
subformula of the formula $\Theta(\bar x,\bar y,\bar z,\bar u)$.
Since $\bar 0/\alpha+\bar b/\alpha=\bar b/\alpha$ then $t_i\bar
b(i)=s_i\bar 0(j)$. Since
\begin{equation}\tag{11}
\bar 0/\alpha+\bar c/\alpha=\bar c/\alpha+\bar 0/\alpha=\bar
c/\alpha.
\end{equation}
then $t_i\bar c(i)=s_i\bar 0(j)=s_i\bar c(j)$. So $t_i\bar
b(i)=s_i\bar c(j)$. Suppose $t_iz_i=s_iy_j$ is an atomic
subformula of the formula $\Theta(\bar x,\bar y,\bar z,\bar u)$.
Since
\begin{equation}\tag{12}
\bar b/\alpha+\bar 0/\alpha=\bar b/\alpha,
\end{equation}
then $t_i\bar b(i)=s_i\bar 0(j)$. If $t_iz_i=s_iz_j$ is an atomic
subformula of the formula $\Theta(\bar x,\bar y,\bar z,\bar u)$,
then (11) implies the equality $t_i\bar b(i)=s_i\bar a(j)$.
Moreover from (11) and (12) we have
$$\mathcal{C}\models\Psi(\bar c,\bar 0,\bar a)\wedge\Phi_1(\bar c)\wedge \Phi_2(\bar 0)
 \wedge \Phi_3(\bar b)\wedge \Phi_4(\bar a).$$
Hence $\bar c/\alpha+\bar 0/\alpha=\bar b/\alpha$, that is $\bar
c/\alpha=\bar b/\alpha$ and $|X|=1$. Contradiction.
\end{proof}

Lemma \ref{eliminac} implies

\begin{theorem}\label{antiadd} If $S$ is a commutative monoid, the class $\mathfrak {R}$ of regular $S$-acts is
axiomatizable and primitive normal then the class $\mathfrak {R}$
is antiadditive.
\end{theorem}

By Theorems \ref{R prim norm iff}, \ref{antiadd} and definition of
antiadditive class we have

\begin{corollary} Let $S$ be a commutative monoid and the class $\mathfrak {R}$ of regular $S$-acts is
axiomatizable.
Then the following conditions are equivalent:

1) the class $\mathfrak {R}$ is primitive normal;

2) the class $\mathfrak {R}$ is antiadditive;

3) the semigroup $R$ is linearly order.
\end{corollary}

\end{document}